\documentclass[11pt]{amsart}
\usepackage[margin=1in]{geometry}

\usepackage{hyperref}
\usepackage{amsmath, amsthm, amssymb}
\usepackage[backend=biber]{biblatex}
\newtheorem{thm}{Theorem}
\newtheorem{lem}{Lemma}
\newtheorem{cor}{Corollary}

\newtheorem{prp}{Proposition}

\theoremstyle{definition}
\newtheorem{dfn}{Definition}

\theoremstyle{remark}

\newtheorem{ex}{Example}

\renewcommand{\sigma}{\varsigma}
\renewcommand{\epsilon}{\varepsilon}
\renewcommand{\phi}{\varphi}

\newcommand{\cleq}{\leq_{\mathrm{c}}}

\newcommand{\ceq}{\sim_{\mathrm{c}}}
\newcommand{\cleqq}{\leq_{[\mathrm{c}]}}

\newcommand{\tr}{{\negthickspace \top \negthickspace}}

\author{Robert Angarone}
\address{University of Minnesota\\
206 Church Street SE \\ 55455 Minneapolis, Minnesota, USA}
\email{angar017@umn.edu}

\author{Daniel Soskin}
\address{Institute for Advanced Study\\
1 Einstein Drive\\08540 Princeton, New Jersey, USA}
\email{dsoskin@ias.edu}
\title{Generalized Diagonals in Positive Semi-Definite Matrices}

\begin{document}

\begin{abstract}
We describe all inequalities among generalized diagonals in positive semi-definite matrices. These turn out to be governed by a simple partial order on the symmetric group. This provides an analogue of results of Drake, Gerrish, and Skandera on inequalities among generalized diagonals in totally nonnegative matrices.
\end{abstract}

\maketitle

\section{Introduction}
A matrix $X \in M_{n}(\mathbb C)$ is called \emph{Hermitian} if it satisfies $X^*=X$, where $X^*$ denotes the conjugate transpose of $X$. A matrix is called \emph{positive semi-definite} (PSD) if it is Hermitian and also satisfies $z^* X z \geq 0$ for all $z \in \mathbb C^n.$ If in fact $X \in M_{n}(\mathbb R)$, the Hermitian property reduces to symmetry $X^{\tr} = X$, and the positive semi-definite property reduces to $z^{\tr} X z \geq 0$ for all $z \in \mathbb R^n$.

Given a class, such as the above, of matrices defined by positivity conditions, it is common to ask which \textit{additional} inequalities the entries of all matrices in that class satisfy. For example, another well-studied class of matrices defined by positivity conditions is that of \emph{totally nonnegative} (TNN) matrices, i.e. those matrices in $M_{n}(\mathbb R)$ for which each minor is nonnegative. In \cite{DGS}, Drake, Gerrish and Skandera showed that for two permutations $\tau, \sigma$, the following inequality
\begin{equation}\label{eq:dgs}
x_{1,\sigma(1)}x_{2,\sigma(2)}\dots x_{n,\sigma(n)} \leq x_{1,\tau(1)}x_{2,\tau(2)}\dots x_{n,\tau(n)}    
\end{equation}
holds for any TNN matrix $X=(x_{i,j})_{i,j=1}^n$ if and only if $\tau \leq \sigma$ in Bruhat order. 

Our goal in the present work is to study inequalities of the form~(\ref{eq:dgs}) for PSD matrices. The close ties between inequalities for TNN and PSD matrices are presented in \cite{fallat2011total}. Yet there are some inequalities, such as~(\ref{eq:dgs}), which have been studied for one of these classes, but not the other; see \cite{SkanSoskinBJArx} for recent results and open questions. For a general PSD matrix, which may have complex entries, one or both products in inequality~(\ref{eq:dgs}) may be complex, so inequality~(\ref{eq:dgs}) may not be well-defined. As such, we begin by describing inequalities in PSD matrices of the form
\begin{equation}\label{eq:|dgs|}
|x_{1,\sigma(1)}x_{2,\sigma(2)}\dots x_{n,\sigma(n)}| \leq |x_{1,\tau(1)}x_{2,\tau(2)}\dots x_{n,\tau(n)}|.
\end{equation}
We will then be able to derive inequalities that hold without absolute values, as well as all inequalities for real-valued PSD matrices, as corollaries. These are described in terms of a simple, combinatorially-defined partial order on the symmetric group.

The Hermitian property will force some of the products $x_{1,\sigma(1)}x_{2,\sigma(2)}\dots x_{n,\sigma(n)}$ to be equal. Thus we also give a combinatorial description of the conditions under which two such products are equal. The resulting equivalence relation on the symmetric group, when combined with the partial order mentioned in the previous paragraph, gives rise to a complete description of the relative order of these products.

In Section~\ref{sec:background}, we set up some notation and recall essential background information. In Section~\ref{sec:partial-order}, we describe the partial order on $S_n$ which will be used to state the main results. In Section~\ref{sec:main}, we state and prove our main results, including corollaries for real-valued cases.

\section*{Acknowledgments}
The authors would like to thank Pavlo Pylyavskyy for numerous productive discussions. This material is based upon work supported by the National Science Foundation under Grant No. DMS-1926686 and Grant No. DMS-1949896.

\section{Background and Notation}\label{sec:background}


\begin{dfn}\label{dfn:gen-diag}
Fix the following notation:
    \[\begin{tabular}{rl}
    $S_n$ :& the symmetric group on $n$ letters, \\
    $M_n(\mathbb{F})$ :& the set of $n \times n$ matrices with entries in $\mathbb{F}$, \\
    $x_{ij}$ :& the $(i,j)$-th entry of the matrix $X$.
    \end{tabular}\]
In addition, for any $X \in M_n(\mathbb{C})$ and $\sigma \in S_n$, the \textit{generalized diagonal of $X$ corresponding to $\sigma$}, or simply \textit{$\sigma$-diagonal of $X$}, is defined to be
    \[X_\sigma = \prod_{k=1}^n x_{k, \sigma(k)}.\]
\end{dfn}

Next, we recall several equivalent formulations of the positive semi-definite (PSD) property. The main results of this paper will also specialize to positive \textit{definite} (PD) matrices, which are a strict subclass of PSD matrices. Thus we recall that notion here as well.

\begin{thm}\label{thm:PSD-equivalents}
The following are equivalent for $X \in M_n(\mathbb{C})$:
    \begin{enumerate}
    \item $X=X^*$ and $z^* X z \geq 0$ for all $z \in \mathbb C^n$,
    \item there exists some $B \in M_n(\mathbb{C})$ such that $X = B \cdot B^*$, and
    \item $X=X^*$ and all principal minors of $X$ are nonnegative.
    \end{enumerate}
A matrix $X \in M_n(\mathbb{C})$ is called \textit{positive definite} if one of the following equivalent conditions holds:
    \begin{enumerate}
    \item $X=X^*$ and $z^* X z > 0$ for all $z \in \mathbb C^n$,
    \item there exists some invertible $B \in M_n(\mathbb{C})$ such that $X = B \cdot B^*$, and
    \item $X=X^*$ and all principal minors of $X$ are positive.
    \end{enumerate}
\end{thm}

Positive semi-definite matrices are but one of many classes of `positive' matrices. Of particular interest to combinatorialists is the class of totally nonnegative (TNN) matrices, i.e. those matrices for which \textit{every} minor is nonnegative. This condition is a significant strengthening of item (iii) in the definitions above; however, we do not require symmetry for TNN matrices.

Work of Drake, Gerrish, and Skandera \cite{DGS} establishes that generalized diagonals in TNN matrices are ordered according to the Bruhat order. This well-studied partial order on the symmetric group has a rich combinatorial structure, and is the subject of many deep results and conjectures. Here we present one of the many characterizations of Bruhat order, found in \cite{Humphreys}; we also refer the reader to \cite{BB} for more background material on Bruhat order.

\begin{dfn}\label{dfn:bruhat-order}[\cite{Humphreys}, p. 119]
Define a partial order on $S_n$ via the following rule: given two permutations $\sigma, \tau \in S_n$, we say $\tau \leq \sigma$ \textit{in (strong) Bruhat order} if $\sigma$ can be obtained from $\tau$ by a sequence of transpositions $(i,j)$ with $i<j$ and $i$ occurring before $j$ when $\tau$ is written in one-line notation.
\end{dfn}

\begin{thm}[\cite{DGS}]
If $X \in M_n(\mathbb{R})$ is totally nonnegative and $\sigma,\tau \in S_n$ are permutations with $\tau \leq \sigma$ in Bruhat order, then $X_\sigma \leq X_\tau$.
\end{thm}

\section{A Partial Order on the Symmetric Group}\label{sec:partial-order}

\begin{dfn}\label{dfn:cycle-notation}
For a permutation $\sigma \in S_n$ and $\ell \geq 2$, let $C_\ell(\sigma)$ be the set of $\ell$-cycles of $\sigma$. Let $C_1(\sigma)$ be the set of fixed points of $\sigma$. Let $C(\sigma) = C_1(\sigma) \cup C_2(\sigma) \cup \ldots \cup C_n(\sigma)$.
\end{dfn}

\begin{dfn}\label{dfn:cleq}
Define a partial order called \textit{cycle inclusion order} on $S_n$ as follows: $\tau \cleq \sigma$ if $C_k(\tau) \subseteq C_k(\sigma)$ for all $k \geq 2$.
\end{dfn}

\begin{ex}\label{ex:cleq}
We have $(132) \cleq (132)(45)$. Letting $\textrm{id}$ denote the identity permutation, we have (vacuously) that $\textrm{id} \cleq \sigma$ for any $\sigma \in S_n$. The permutations $(123)$ and $(132)$ are incomparable in cycle inclusion order. 
\end{ex}


\begin{prp}\label{prp:relationship-to-bruhat}
If $\tau \cleq \sigma$, then $\tau \leq \sigma$ in Bruhat order.
\end{prp}

\begin{proof}
Since $C_\ell(\tau) \subseteq C_\ell(\sigma)$ for all $\ell \geq 2$, we may write
\begin{align*}
\sigma &= \prod_{\ell \geq 2} \prod_{\pi \in C_\ell(\sigma)} \pi \\
&=  \prod_{\ell \geq 2} \left(\prod_{\kappa \in C_\ell(\tau)} \kappa \prod_{\pi \in C_\ell(\sigma) \setminus C_\ell(\tau)} \pi \right) \\
&= \tau \left(\prod_{\ell \geq 2} \prod_{\pi \in C_\ell(\sigma) \setminus C_\ell(\tau)} \pi\right).
\end{align*}
By Definition~\ref{dfn:bruhat-order}, it will suffice to show the following claim: each $\pi \in C_\ell(\sigma) \setminus C_\ell(\tau)$ can be obtained as a sequence of transpositions $(i,j)$ with $i < j$ and $i$ appearing before $j$ when $\tau$ is written in one-line notation.

Suppose $\pi \in C_{\ell}(\sigma) \setminus C_{\ell}(\tau)$, and write $\pi = (c_1, \ldots, c_\ell)$. We claim each $c_i$ is fixed by $\tau$. Otherwise, if $\tau(c_i) \neq c_i$, then there is some cycle of $\tau$ involving $c_i$. That cycle is either equal to $\pi$ or not a cycle of $\sigma$, since the cycles of $\sigma$ are disjoint. This would be a contradiction in either case. Since each $c_i$ is fixed by $\tau$, the numbers $c_1, \ldots, c_\ell$ appear in increasing order when $\tau$ is written in one-line notation. Thus the claim is reduced to the fact that $\pi$ can be obtained as a sequence of transpositions $(i,j)$ with $i < j$ and $i,j \in \{c_1,\ldots,c_\ell\}$. This is possible for any cycle.
\end{proof}

\section{Inequalities for Positive Semi-Definite Matrices}\label{sec:main}

\begin{prp}\label{prp:cleqq-implies-inequality}
Suppose $X \in M_n(\mathbb{C})$ is positive semi-definite, and that $\sigma, \tau \in S_n$. If $\tau \cleq \sigma$, then $|X_\sigma| \leq |X_\tau|$.
\end{prp}

\begin{proof}
The key observation is that, by Theorem~\ref{thm:PSD-equivalents}, the principal minors of $X$ are nonnegative. In particular, for any $i,j \in [n]$, we have $0 \leq x_{ii}x_{jj} - x_{ij}x_{ji}$. Since $x_{ii}x_{jj}$ is real, and we know $x_{ij}x_{ji} = x_{ij}\overline{x_{ij}}$ is real, we may re-write this as $x_{ij}x_{ji} \leq x_{ii}x_{jj}$. In fact, since both sides of the equation are positive, we can take absolute values to obtain $|x_{ij}||x_{ji}| \leq |x_{ii}||x_{jj}|$. Since $x_{ij}=\overline{x_{ji}}$, and therefore $|x_{ij}|=|x_{ji}|$, the inequality can be rewritten as
    \begin{equation}\label{eq:key-2by2}
    |x_{ij}|^2 \leq |x_{ii}||x_{jj}|, \quad \text{i.e.} \quad |x_{ij}| \leq \sqrt{|x_{ii}x_{jj}|}.
    \end{equation}
Thus, for any cycle $\pi = (c_1, \ldots, c_\ell)$ we have the following inequality:
\begin{align}
    |x_{c_1,c_2}x_{c_2,c_3}\ldots x_{c_{\ell},c_1}| &\leq \sqrt{|x_{c_1,c_1}||x_{c_2,c_2}|}\cdot\sqrt{|x_{c_2,c_2}||x_{c_3,c_3}|}\cdots\sqrt{|x_{c_\ell,c_\ell}||x_{c_1,c_1}|}  \\
    &= |x_{c_1,c_1}|\cdot|x_{c_2,c_2}|\cdots|x_{c_{\ell},c_{\ell}}|. \label{eq:one-cycle-ineq}
\end{align}
For any $\tau \cleq \sigma$, observe that
\begin{align*}
    |X_\sigma| &= \left(\prod_{\ell \geq 2} \prod_{\substack{\pi \in C_\ell(\sigma) \\ \pi = (c_1, \ldots, c_\ell)}} |x_{c_1,c_2}x_{c_2,c_3}\cdots x_{c_{\ell},c_1}|\right) \left(\prod_{i \in C_1(\sigma)} |x_{ii}| \right) \\
    &= \left(\prod_{\ell \geq 2} \prod_{\substack{\pi \in C_\ell(\tau) \\ \pi = (c_1, \ldots, c_\ell)}} |x_{c_1,c_2}x_{c_2,c_3}\cdots x_{c_{\ell},c_1}| \prod_{\substack{\kappa \in C_\ell(\sigma) \setminus C_\ell(\tau) \\ \kappa = (c_1, \ldots, c_\ell)}} |x_{c_1,c_2}x_{c_2,c_3}\cdots x_{c_{\ell},c_1}|\right) \left(\prod_{i \in C_1(\sigma)} |x_{ii}| \right).\\
\intertext{Applying Inequality~\ref{eq:one-cycle-ineq} to each $\kappa \in C_\ell(\sigma)\setminus C_\ell(\tau)$, we obtain}
    &\leq \left(\prod_{\ell \geq 2} \prod_{\substack{\pi \in C_\ell(\tau) \\ \pi = (c_1, \ldots, c_\ell)}} |x_{c_1,c_2}x_{c_2,c_3}\ldots x_{c_{\ell},c_1}|\right) \left(\prod_{\substack{\kappa \in C_\ell(\sigma) \setminus C_\ell(\tau) \\ \kappa = (c_1, \ldots, c_\ell)}} |x_{c_1,c_1}x_{c_2,c_2}\cdots x_{c_{\ell},c_{\ell}}|\right) \left(\prod_{i \in C_1(\sigma)} |x_{ii}| \right). \\
\intertext{As was argued in the proof of Proposition~\ref{prp:relationship-to-bruhat}, if $\kappa = (c_1, \ldots, c_\ell) \in C_\ell(\sigma) \setminus C_\ell(\tau)$, then each of the $c_i$ is in fact a fixed point of $\tau$. Thus we obtain}
    &= \left(\prod_{\ell \geq 2} \prod_{\substack{\pi \in C_\ell(\tau) \\ \pi = (c_1, \ldots, c_\ell)}} |x_{c_1,c_2}x_{c_2,c_3}\ldots x_{c_{\ell},c_1}|\right) \left(\prod_{i \in C_1(\tau)} |x_{ii}| \right) = |X_\tau|,
\end{align*}
which completes the proof.
\end{proof}

We observe that the condition $X=X^{*}$ implies more relations on the generalized diagonals of $X$ that may not come from Proposition~\ref{prp:cleqq-implies-inequality}. In particular, for any cycle $\pi = (c_1, \ldots, c_\ell)$ we have that $|x_{c_1,c_2}x_{c_2,c_3}\ldots x_{c_{\ell},c_1}|=|x_{c_2,c_1}x_{c_3,c_2}\ldots x_{c_{1},c_{\ell}}|$. To account for these cases, we introduce an equivalence relation on $S_n$.

\begin{dfn}\label{dfn:ceq}
Define an equivalence relation $\ceq$ on $S_n$ as follows: $\sigma \ceq \tau$ if each cycle of $\sigma$ is either equal to, or the inverse of, a cycle of $\tau$ and vice versa. Let $[\sigma]$ denote the equivalence class of $\sigma$ under $\ceq$.
\end{dfn}

\begin{ex}
We have $(123)(45)(678) \ceq (321)(45)(678) \ceq (123)(45)(876) \ceq (321)(45)(876)$. We have $(132)(456) \not\ceq (123)(456)$ and $(123)(45) \not\ceq (321)$.
\end{ex}

\begin{prp}\label{prp:ceq-implies-equal}
Suppose $X \in M_n(\mathbb{C})$ is positive semi-definite and that $\sigma, \tau \in S_n$. If $\tau \ceq \sigma$, then $|X_\tau| = |X_\sigma|$.
\end{prp}

\begin{proof}
By Definition~\ref{dfn:ceq}, $\tau \ceq \sigma$ implies the existence of a bijective map between cycles in $C(\sigma)$ and cycles in $C(\tau)$ which takes cycles either to themselves or to their inverses. For any cycle $\pi = (c_1, \ldots, c_\ell)$ we have that $|x_{c_1,c_2}x_{c_2,c_3}\ldots x_{c_{\ell},c_1}|=|x_{c_2,c_1}x_{c_3,c_2}\ldots x_{c_{1},c_{\ell}}|$, since $X$ is Hermitian. Thus the above map preserves the absolute values of the products of matrix entries corresponding to each cycle, which implies that $|X_\tau| = |X_\sigma|$.
\end{proof}

\begin{ex}
Recall that $(123)(45)(678) \ceq (321)(45)(678) \ceq (123)(45)(876) \ceq (321)(45)(876)$. 
We have \[|X_{(123)(45)(678)}|=|X_{(321)(45)(678)}|=|X_{(123)(45)(876)}|=|X_{(321)(45)(876)}|.\]
Indeed, the equalities above follow from observations that \[|x_{1,2}x_{2,3}x_{3,1}|=|x_{1,3}x_{2,1}x_{3,2}| \quad \text{and} \quad |x_{6,7}x_{7,8}x_{8,6}|=|x_{6,8}x_{7,6}x_{8,7}|.\]
\end{ex}

Next, we introduce a partial order which describes the inequalities from Proposition~\ref{prp:cleqq-implies-inequality} and the equalities from Proposition~\ref{prp:ceq-implies-equal} simultaneously. 

\begin{dfn}\label{dfn:cleqq}
Define a partial order $\cleqq$ on the equivalence classes of $S_n$ under $\ceq$ as follows: $[\tau] \cleqq [\sigma]$ if there exist $\tau' \in [\tau]$ and $\sigma' \in [\sigma]$ such that $\tau' \cleq \sigma'$.
\end{dfn}

\begin{ex}
We have $[(123)] \cleqq [(321)(45)]$, even though $(123) \not\cleq (321)(45)$. This is because $(321) \in [(123)]$, and $(321) \cleq (321)(45)$. This comparison allow us to obtain additional relations among generalized diagonals for a positive semi-definite matrix $X$. Applying Proposition~\ref{prp:cleqq-implies-inequality}, we know that $ 
|X_{(321)(45)}|  \leq  |X_{(321)}|$; then, applying Proposition~\ref{prp:ceq-implies-equal}, we obtain $|X_{(321)(45)}|  \leq  |X_{(123)}|$.
\end{ex}


\begin{prp}\label{prp:ceqq-well-defined}
The relation in Definition~\ref{dfn:cleqq} is well-defined in the following senses:
\begin{enumerate}
\item The relation $\cleqq$ is indeed a partial order on the $\ceq$-equivalence classes in $S_n$; and
\item if $[\tau] \cleqq [\sigma]$, then for \textit{each} $\tau' \in [\tau]$, there is some $\sigma' \in [\sigma]$ such that $\tau' \cleq \sigma'$. 
\end{enumerate}
\end{prp}

\begin{proof}
These properties are straightforward from Definition~\ref{dfn:cleqq}.
\end{proof}



Our goal in the remainder of this section is to demonstrate that the inequalities in Proposition~\ref{prp:cleqq-implies-inequality} and the equalities in Proposition~\ref{prp:ceq-implies-equal}, along with their immediate corollaries, are the only comparisons among absolute values of generalized diagonals which hold for \textit{all} positive semidefinite matrices. To that end, Proposition~\ref{prp:non-cleqq-counterexample} constructs a counterexample matrix for any additional inequality one might propose. Indeed, the counterexample is a real-valued, positive \textit{definite} matrix. This also demonstrates that one does not obtain additional relations with absolute values when restricting to this special subclass. To do this, we require one additional lemma.

\begin{lem}\label{lem:entries-to-zero}
For any pair $p,q \in [n]$ and any $\epsilon > 0$, there exists a positive definite matrix $A$ such that $a_{pq}=a_{qp}=\epsilon$ while all other entries of $A$ are greater than $1$.
\end{lem}

\begin{proof}
By Definition~\ref{thm:PSD-equivalents}, to construct a positive definite matrix $A$, it suffices to write $A = BB^*$ for some invertible matrix $B$. Thinking of the rows of $B$ as a set of linearly independent vectors $\{v_1, \ldots, v_n\}$, we have $a_{ij} = (BB^*)_{ij} = \left\langle v_i, v_j \right\rangle$.

Thus we must construct a linearly independent set $\{v_1, \ldots, v_n\}$ of vectors such that $\left\langle v_p, v_q \right\rangle > 0$ is arbitrarily small while all other $\left\langle v_r, v_s \right\rangle$ remain bounded below by a positive constant. Letting $e_1, \ldots, e_n$ be the standard basis vectors, we can construct such a collection as follows:
    \[v_k = 
    \begin{cases}
    e_k + \left(\frac{\epsilon}{2}\right)e_q & \text{if } k = p \\
    e_k + \left(\frac{\epsilon}{2}\right)e_p & \text{if } k = q \\
    e_k + e_p + e_q & \text{otherwise}\\
    \end{cases}.\]
Observe that
    \[\left\langle v_p, v_q \right\rangle = \left\langle e_p + \left(\frac{\epsilon}{2}\right)e_q, \,\, e_q + \left(\frac{\epsilon}{2}\right)e_p \right\rangle = \left\langle e_p, \left(\frac{\epsilon}{2}\right)e_p \right\rangle + \left\langle e_q, \left(\frac{\epsilon}{2}\right)e_q \right\rangle = \epsilon.\]
On the other hand, if $r=p$ and $s \neq q$, we have
    \[\left\langle v_r, v_s \right\rangle = \left\langle e_p + \left(\frac{\epsilon}{2}\right)e_q, \,\, e_s + e_p + e_q \right\rangle = \left\langle e_p,e_p \right\rangle + \left\langle \left(\frac{\epsilon}{2}\right)e_q, e_q \right\rangle = 1 + \frac{\epsilon}{2} > 1;\]
similarly for $r=q$ and $s \neq p$. Lastly, if $r$ and $s$ are both not equal to $p$ or $q$, we have
    \[\left\langle v_r, v_s \right\rangle = \left\langle e_r + e_p + e_q, \,\, e_s + e_p + e_q \right\rangle = \left\langle e_p,e_p \right\rangle + \left\langle e_q,e_q\right\rangle = 2.\]
Thus this collection has the desired properties, and finishes the proof.
\end{proof}

\begin{prp}\label{prp:non-cleqq-counterexample}
If $\sigma, \tau \in S_n$ and $[\tau]$ and $[\sigma]$ are incomparable with respect to $\cleqq$, then there exist positive definite $A, A' \in M_n(\mathbb{R})$ such that $0 < A_\sigma < A_\tau$ and $0 < A'_\tau < A'_\sigma$.
\end{prp}

\begin{proof}

If $[\tau]$ is incomparable to $[\sigma]$ with respect to $\cleqq$, we claim that one of the following statements holds:
\begin{enumerate}
\item there exist $p,q$ such that $a_{pq}$ appears as a factor in $A_\sigma$, but neither $a_{pq}$ nor $a_{qp}$ appears in $A_\tau$; or \label{3cyclecase}
\item there exist $p,q$ such that both $a_{pq}$ and $a_{qp}$ appear as a factors in $A_\sigma$, while at most one of $a_{pq}$ and $a_{qp}$ appear as factors in $A_\tau$. \label{2cyclecase}
\end{enumerate}
In either case, we will be able to apply Lemma~\ref{lem:entries-to-zero} to complete the proof.

Since $[\tau]$ and $[\sigma]$ are incomparable with respect to $\cleqq$, there is some cycle $c$ of $\sigma$ which is neither equal to nor the reverse of any cycle of $\tau$. For Case~\ref{3cyclecase}, suppose we have $c = (c_1, c_2, \ldots, c_\ell)$ with $\ell \geq 3$. Then $a_{c_1,c_2}, a_{c_2,c_3}, \ldots, a_{c_\ell, c_1}$ all appear as factors in $A_\sigma$. Certainly, we cannot have all of these factors appearing in $A_\tau$, because then $c$ would be a cycle of $\tau$ as well. Similarly, we cannot have the reverse factors $a_{c_2,c_1}, a_{c_3,c_2}, \ldots, a_{c_1,c_\ell}$ all appearing, because then the reverse of $c$ is a cycle of $\tau$. It remains to eliminate the case wherein, for each $i$, either $a_{c_i, c_{i+1}}$ or $a_{c_{i+1},c_i}$ appears, in some combination of original and reversed factors; here we take indices mod $\ell$. If this were true, then at some point we would have an original factor $a_{c_i,c_{i+1}}$ followed by a reversed factor $a_{c_{i+2},c_{i+1}}$ appearing in $A_\tau$, or vice versa: $a_{c_{i+1}, c_i}$ appears followed by $a_{c_{i+1},c_{i+2}}$. Since $\ell \geq 3$, we have $c_i \neq c_{i+2}$. In this first case, we would have $\tau(c_i)=c_{i+1}$ and and also $\tau(c_{i+2})=c_{i+1}$. This is a contradiction, since $\tau$ is a bijective map and in this case would fail to be injective. In the second case, then we have $\tau(c_{i+1})=c_i$ and $\tau(c_{i+1})=c_{i+2}$, so $\tau$ fails to be well-defined. The only remaining possibility is that, for some $i$, neither $a_{c_i,c_{i+1}}$ nor $a_{c_{i+1},c_i}$ appears in $A_\tau$. Taking $c_i=p$ and $c_{i+1}=q$, we have satisfied Case~\ref{3cyclecase}.

Now, for Case~\ref{2cyclecase}, suppose that the only cycles of $\sigma$ not appearing in $\tau$ are $2$-cycles. Then, letting $c=(c_1,c_2)$ be such a cycle, we know $a_{c_i,c_{i+1}}$ and $a_{c_{i+1},c_i}$ both appear in $A_\sigma$. If both of these appear as factors in $A_\tau$, then $\tau$ has $c$ as a $2$-cycle as well, which is a contradiction. Thus, taking $c_1=p$ and $c_2=q$, we have satisfied Case~\ref{2cyclecase}.

Now, in either case, let us apply Lemma~\ref{lem:entries-to-zero} to construct a positive definite matrix $A$ which has $a_{pq}$ and $a_{qp}$ equal to $\epsilon$, while all other entries are greater than $1$. In Case~\ref{3cyclecase}, we know that $A_\tau > 1$, since it is a product of factors which are each greater than $1$. On the other hand, $A_\sigma$ has a factor equal to $\epsilon$, so for a sufficiently small choice of $\epsilon$, we have $0 < A_\sigma < A_\tau$. Case~\ref{2cyclecase} is similar, except $A_\tau$ may have a factor of $\epsilon$; nevertheless, $A_\sigma$ now has $\epsilon^2$ as a factor, so again we may choose $\epsilon$ sufficiently small so that $0 < A_\sigma < A_\tau$. In either case, we have obtained the desired inequality.

To finish the proof, we may repeat the same argument with $\sigma$ and $\tau$ swapped and construct a matrix $A'$ with $0 < A'_\tau < A'_\sigma$.
\end{proof}

To conclude, we provide one Theorem and two Corollaries which summarize the results in this section.

\begin{thm}\label{thm:abs-complex}
Let $\sigma, \tau \in S_n$. Then the inequality $|X_\sigma| \leq |X_\tau|$ holds for all positive semi-definite $X \in M_n(\mathbb{C})$ if and only if $[\tau] \cleqq [\sigma]$, with equality $|X_\sigma| = |X_\tau|$ for all such $X$ if and only if $\tau \ceq \sigma$.
\end{thm}

\begin{proof}
First, we deal with the statement on inequalities. The `if' direction is obtained by applying Proposition~\ref{prp:ceq-implies-equal} to Proposition~\ref{prp:cleqq-implies-inequality}. The `only if' direction is Proposition~\ref{prp:non-cleqq-counterexample}.

For the equalities: if $\tau \not\ceq \sigma$, then one of the permutations has a cycle which neither equal to, nor the reverse of, any cycle of the other. Then we may apply an argument nearly identical to the one found in the proof of Proposition~\ref{prp:non-cleqq-counterexample} to obtain positive definite matrices with $A_\sigma \neq A_\tau$. If $\tau \ceq \sigma$, then we apply Proposition~\ref{prp:ceq-implies-equal}.
\end{proof}

As was mentioned before, one obtains no additional relations of the form $|X_\sigma| \leq |X_\tau|$ or $|X_\sigma| = |X_\tau|$ if we insist on $X$ being real-valued or positive definite. However, once we drop the absolute values, we are able to make some additional statements, mostly owing to the following lemma:

\begin{lem}\label{lem:inv-implies-pos}
Suppose $X \in M_n(\mathbb{C})$ is positive semi-definite, and that $\tau \in S_n$ is an involution. Then $X_\tau$ is a positive real number.
\end{lem}

\begin{proof}
Since $\tau$ is an involution, $X_\tau$ is a product of factors of the form $x_{ij} x_{ji} = x_{ij}\overline{x_{ij}}$. For any complex number $z$, we have $z\overline{z} \in \mathbb{R}$ and $z\overline{z} > 0$.
\end{proof}

\begin{cor}\label{cor:nabs-complex}
Let $\sigma, \tau \in S_n$. Then the inequality $X_\sigma \leq X_\tau$ holds for all positive semi-definite $X \in M_n(\mathbb{C})$ if and only if $\sigma$ and $\tau$ are involutions and $\tau \cleq \sigma$.
\end{cor}

\begin{proof}
For the `if' direction, since $\sigma$ and $\tau$ are both involutions, $X_\sigma$ and $X_\tau$ are positive by Lemma~\ref{lem:inv-implies-pos}, so if $\tau \cleq \sigma$, we apply Proposition~\ref{prp:cleqq-implies-inequality} and obtain \[X_\sigma = |X_\sigma| \leq |X_\tau| = X_\tau.\] For the `only if' direction, note that either $\sigma$ and $\tau$ are not involutions, in which case the inequality is not well-defined, or else $\tau$ and $\sigma$ are not comparable with respect to $\cleq$. Since an involution is composed only of $2$-cycles, and each $2$ cycle is equal to its inverse, each $\ceq$-equivalence class of an involution is a singleton. Thus the orders $\cleq$ and $\cleqq$ are the same on involutions, and we may apply Proposition~\ref{prp:non-cleqq-counterexample} to obtain a counterexample.
\end{proof}

\begin{cor}\label{cor:nabs-real}
Let $\sigma, \tau \in S_n$. Then the inequality $X_\sigma \leq X_\tau$ holds for all positive semi-definite $X \in M_n(\mathbb{R})$ if and only if $\tau$ is an involution and $\tau \cleq \sigma$.
\end{cor}

\begin{proof}
Since $\tau$ is an involution, by Lemma~\ref{lem:inv-implies-pos}, we know $X_\tau > 0$. Thus, when $\tau \cleq \sigma$, we have $X_\sigma \leq |X_\sigma| \leq |X_\tau|=X_\tau$ by Proposition~\ref{prp:cleqq-implies-inequality}. For the other direction, we may apply a similar argument as in the proof of Corollary~\ref{cor:nabs-complex}.
\end{proof}






\printbibliography

\end{document}